\documentclass[11pt]{amsart}

\usepackage{geometry}
\usepackage{amsmath}
\usepackage{amssymb}

\newtheorem{theorem}{Theorem}
\newtheorem{corollary}{Corollary}
\newtheorem{lemma}{Lemma}

\theoremstyle{remark}
\newtheorem{example}{Example}
\newtheorem*{remark}{Remark}

\renewcommand{\leq}{\leqslant}
\renewcommand{\geq}{\geqslant}

\title{Irreducible polynomials over a finite field with restricted coefficients}

\author[S. Porritt]{Sam Porritt}
\address{Department of Mathematics\\University College London\\
25 Gordon Street, London, England}
\email{samuel.porritt.15@ucl.ac.uk}

\begin{document}
\maketitle

\begin{abstract}
We prove a function field analogue of Maynard's celebrated result about primes with restricted digits. That is, for certain ranges of parameters the parameters $n$ and $q$, we prove an asymptotic formula for the number of irreducible polynomials of degree $n$ over a finite field $\mathbb{F}_q$ whose coefficients are restricted to lie in a given subset of $\mathbb{F}_q$.
\end{abstract}

%\part{Use this type of header for very long papers only}
% use lowercase except for proper names

\section{Introduction}

Many theorems concerning the existence of irreducible polynomials over a finite field of a special form have been proved. A discussion of such results can be found in~\cite{TuWa}. In this paper we will prove a function field analogue of a result of Maynard~\cite{May} concerning primes with missing digits. He proved that for large enough integers $b$, the primes have the expected asymptotic density inside those integers which can be written in base $b$ using only certain specified digits. We will prove the following natural analogue for polynomials in $\mathbb{F}_q[t]$.

\begin{theorem} \label{thm}
Let $\mathcal{R} \subset \mathbb{F}_q$ be a subset of size $s$ and assume that $s$ is less than $\sqrt{q}/2$. Suppose that $q \geq 500 $ and $ n \geq 100(\log q)^2 $. The number of irreducible, monic polynomials of degree $n$ with coefficients only from $\mathbb{F}_q \backslash \mathcal{R}$ (except possibly the leading 1) is given by
\[\frac{ q }{q-1}\frac{(q-s)^n}{n}\left( \Lambda + O\left( q^{-n^{1/2}/7} \right) \right),\]
where
\[ \Lambda = \begin{cases}
1 & \text{ if } 0 \in \mathcal{R} \\
1-\frac{1}{q-s} & \text{ if } 0 \notin \mathcal{R}.
\end{cases} \]
\end{theorem}

\begin{remark} Beyond stipulating that $s \leq \sqrt{q}$, the constraints on the sizes of $s,q$ and $n$ are somewhat artificial, and were chosen with the aim of producing a more presentable error term. A more complicated, but more widely applicable error term is presented at the end of section 4 from which the following two examples follow.
\end{remark}

\begin{example}
In the special case of $s=1$, we get an asymptotic formula for any $q \geq 17$. In particular, we show that the number of irreducible polynomials of degree $n$ with a single coefficient from $\mathbb{F}_{17}$ unavailable is asymptotic to $\Lambda\frac{16}{17}(16)^n/n$ as $n \rightarrow \infty$.
\end{example}

\begin{example}
An asymptotic formula still holds in the case of fixed $n$ and $q \rightarrow \infty$ provided that $s = o(q^{1/2}).$
\end{example}

As in the integer setting, we can take $s$ to be larger when the set $\mathcal{R}$ has additional structure. For example, in section 5 we will prove the following.

\begin{theorem} \label{thm2}
Suppose $\delta > 0$ and $p$ is a prime sufficiently large in terms of $\delta$. Then for any subset $\mathcal{R} = \{r, r+1, \ldots, r+s-1\} \subset \mathbb{F}_p$ of $s$ consecutive coefficients with $p-s > p^{3/4 + \delta}$, the number of irreducible, monic polynomials of degree $n$ with coefficients only from $\mathbb{F}_p \backslash \mathcal{R}$ (except possibly the leading 1) is given by
\[\frac{ p }{p-1}\frac{(p-s)^n}{n}\left( \Lambda + O\left(e^{-c n^{1/2}}\right)\right),\]
for some positive constant $c$ depending on $p$ and $\delta$.
\end{theorem}

The integer version of Theorem~\ref{thm} was proved in~\cite{May} under the assumption that the number of restricted digits $s$ satisfies $s \leq b^{1/4 -\delta}$ and the base $b$ is sufficiently large in terms of $\delta$. An analogue of Theorem~\ref{thm2} was proved under the assumption that $\mathcal{R} = \{0, 1, \ldots, s-1\}$ and $s \leq b - b^{3/4 + \delta}$. The proofs of Theorems~\ref{thm} and~\ref{thm2} will use the circle method over $\mathbb{F}_q[t]$ along the lines of~\cite{Ha} and~\cite{May}. Two features make our arguments substantially simpler. First, we may make use of Weil's Riemann hypothesis for curves over a finite field which gives very good control for exponential sums over irreducibles. Second, we don't have to deal with any technicalities which arise from the fact that sometimes digits are `carried' when rational integers are added. This doesn't happen with polynomials over a finite field.

For an overview of digit related results in the integers, see the recent work of Dietmann, Elsholtz and Shparlinski~\cite{DiElSh} which also contains a section on finite fields, improving an earlier result of  Dartyge, Mauduit and S\'{a}rk\"{o}zy~\cite{DaMaS}. See also~\cite{OpSh}, which contains an extensive list of references to related problems.

\section{Definitions and set up}

This section introduces some notation we will be using. Let $q$ be a prime power and $\mathbb{F}_q$ be the field with $q$ elements and characteristic $p$. Let $\mathcal{R}=\{r_1, \ldots, r_s\}\subset \mathbb{F}_q$ be a subset of \emph{forbidden} coefficients. We are interested in counting monic (sometimes called positive) irreducible polynomials in $\mathbb{F}_q[t]$ of degree $n$, all of whose coefficients, apart from possibly the leading 1, are in the set $\mathcal{R}^c:=\mathbb{F}_q \backslash \mathcal{R}$. The function field analogue of the real numbers is the completion of the field of fractions of $\mathbb{F}_q[t]$ with respect to the norm defined by
\[|f/g| = \begin{cases}
q^{\deg f - \deg g} & \text{ if } f\neq 0 \\
0 & \text{ otherwise .}
\end{cases}\]
This completion is naturally identified with the ring of formal Laurent series $\mathbb{F}_q((1/t))=\{ \; \sum_{i\leq j} x_i t^i \: : \: x_i \in \mathbb{F}_q, \: j \in \mathbb{Z} \; \}.$ The norm defined above is extended to $x = \sum_{ i \leq j } x_i t^i \in \mathbb{F}_q((1/t))$ by setting $|x| = q^j$ where $j$ is the largest index with $x_j \neq 0$. The subscript notation $x_i$ will be used again to refer to the coefficient of $t^i$ in $x$. The analogue of the real unit interval is $\mathbb{T}:=\{ \sum_{i < 0} x_i t^i \: : \: x_i \in \mathbb{F}_q \},$ and is a subring of $\mathbb{F}_q((1/t))$.
Define $\psi : \mathbb{F}_q \rightarrow \mathbb{C}^\times$ by
\[\psi(a) = \exp(2\pi i\text{tr}(a)/p)\]
where $\text{tr} : \mathbb{F}_q \rightarrow \mathbb{F}_p$ is the usual trace map. Define also the additive character $\textbf{e}_q : \mathbb{F}_q((1/t))\rightarrow \mathbb{C}^\times$ by \[\textbf{e}_q (x) = \psi(x_{-1}).\]
Fix a Haar measure on the additive group $\mathbb{T}$ normalised so that $\int_{\mathbb{T}}d x =1$. Then for all $a \in \mathbb{F}_q[t]$, we have
\[\int_{\mathbb{T}}\textbf{e}_q(a x)d x = \begin{cases} 1 & \text{ if } a = 0 \\ 0 & \text{ if } a \neq 0. \end{cases}\]
For $x \in \mathbb{T}$, define the sum over monic irreducible polynomials of degree $n$
\[ S(x) = \sum_{\deg p = n} \textbf{e}_q(p x).\] Let $\mathcal{M}_\mathcal{R}(n)$ be the set of monic polynomials of degree $n$ with non-leading coefficients taken from $\mathcal{R}^c$ and define
\[ S_\mathcal{R}(x) = \sum_{m \in \mathcal{M}_\mathcal{R}(n) } \textbf{e}_q(m x).\] So $S(x)$ and $S_\mathcal{R}(x)$ depend on $n$ even though this is not apparent from the notation. The main quantity of interest, the number of irreducible polynomials in $\mathcal{M}_\mathcal{R}(n)$, is then given by
\begin{equation}
N(\mathcal{R},n) = \int_{\mathbb{T}} S(x) \overline{S_\mathcal{R}(x)} dx.
\end{equation}
We will make use of the important fact that for each $x \in \mathbb{T}$, there exist unique $a, g \in \mathbb{F}_q[t]$ with $g$ monic, $a$ and $g$ coprime, and $|a| < |g| \leq q^{n/2}$ such that
\[ \left| x - \frac{a}{g} \right| < \frac{1}{q^{\deg g + n/2}}.\]
This fact is Lemma 3 from~\cite{Pol}. It implies that we can partition $\mathbb{T}$ into the so-called Farey arcs as
\[\mathbb{T} = \bigcup_{\substack{ |a|<|g| \leq q^{n/2} \\ (a,g)=1 }} \mathcal{F}\left(\frac{a}{g}\:,\:q^{\deg g + n/2}\right)\]
where $\mathcal{F}\left(\frac{a}{g}\:, \:\lambda\right) = \{ x \in \mathbb{T} \: : \: \big| \frac{a}{g} - x \big| < \frac{1}{\lambda} \}.$

As usual we let $\mu(f)$ denote the M\"{o}bius function, defined as $(-1)^k$ if $f$ is the product of $k$ distinct irreducibles and 0 otherwise. Let $\phi(f)$ be the size of the unit group $(\mathbb{F}_q[t]/(f))^\times$, that is $|f|\prod_{ \omega |f}(1-1/|\omega|)$, where the product is over all monic irreducibles dividing $f$. Finally, let $\pi(n)$ be the number of monic, irreducible polynomials of degree $n$ and recall the prime number theorem in the form $\sum_{d|n}d\pi(d) = q^n$.

\section{Lemmas}

The sum $S(x)$ was analysed in~\cite{Hay}. Our first lemma is Lemma 5 in~\cite{Pol} and is a consequence of Weil's Riemann Hypothesis for curves over a finite field.

\begin{lemma} \label{l1}
Let $a,g \in \mathbb{F}_q[t]$ be two polynomials with $(a,g)=1$ and $\gamma \in \mathbb{T}$, satisfying $|a|<|g| \leq q^{n/2}$ and $|\gamma|<1/q^{\deg g + n/2}$. We have
\[ S\left(\frac{a}{g}+\gamma \right)=\frac{\mu(g)}{\phi(g)}\pi(n) \textbf{e}_q(\gamma t^n)\textbf{1}_{|\gamma|<1/q^{n}}+E \]
with $|E| \leq q^{n-\frac{1}{2}[\frac{n}{2}]}$.
\end{lemma}

For a subset $A \subset \mathbb{F}_q$, define the Fourier coefficient $\widehat{\textbf{1}_{A}}(r) := \sum_{n \in A} \psi(nr).$ It turns out that the average value of $|S_\mathcal{R}(x)|$ can be written quite neatly in terms of the Fourier coefficients of the set $\mathcal{R}^c$.

\begin{lemma} \label{l2}
\[\int_{\mathbb{T}} |S_\mathcal{R}(x)|dx = \left(\sum_{r \in \mathbb{F}_q}|\widehat{\textbf{1}_{\mathcal{R}^c}}(r)|/q\right)^n. \]
\end{lemma}

\begin{proof}
First
\begin{align*}
S_\mathcal{R}(x) &= \sum_{m \in \mathcal{M}_\mathcal{R}(n) } \textbf{e}_q(m x) \\
&= \textbf{e}_q(x t^n) \prod_{i=0}^{n-1}\left(\sum_{n_i\in \mathcal{R}^c}\textbf{e}_q(x n_i t^i)\right) \\
&= \textbf{e}_q(x t^n) \prod_{i=0}^{n-1}\left(\sum_{n_i\in \mathcal{R}^c} \psi(n_i x_{-i-1})\right).\\
\end{align*}
Notice that $|S_\mathcal{R}(x)|$ only depends on the leading $n$ coefficients $(x_{-1}, \ldots, x_{-n})$ of $x$ and so, for each $a \in \mathbb{F}_q[t]$, $|S_\mathcal{R}(a/t^n+\gamma)|$ is constant in the range $|\gamma| < 1/q^n$, a set of measure $1/q^n$. Therefore,

\begin{align*}
\int_{\mathbb{T}}\big|S_{\mathcal{R}}(x)\big| d x & = \frac{1}{q^n}\sum_{\deg a < n}\big|S_{\mathcal{R}}\left(\frac{a}{t^n}\right)\big| \\
& = \frac{1}{q^n} \sum_{\deg a < n}\left|\prod_{i=0}^{n-1}\sum_{n_i \in \mathcal{R}^c}\psi(n_i a_{n-i-1})\right| \\
& = \frac{1}{q^n}\sum_{\deg a < n}\prod_{i=0}^{n-1}\left|\widehat{1_{\mathcal{R}^c}}(a_{n-i-1})\right| \\
& = \frac{1}{q^n}\left(\sum_{r\in \mathbb{F}_q}\left|\widehat{\textbf{1}_{\mathcal{R}^c}}(r)\right|\right)^n
\end{align*}
which completes the proof of the lemma.
\end{proof}

\begin{corollary}\label{c1}
\[\int_{\mathbb{T}} |S_\mathcal{R}(x)|dx \leq (\sqrt{s}+1 -2s/q)^n,\]
with equality in the case $s=1$.
\end{corollary}

\begin{proof}
Notice that
\begin{equation*}
\widehat{\textbf{1}_{\mathcal{R}^c}}(r) + \widehat{\textbf{1}_{\mathcal{R}}}(r) = \sum_{n \in \mathbb{F}_q}\psi(rn) =
\begin{cases}
q & \text{ if } r = 0 \\
0 & \text{ if } r \neq 0.
\end{cases}
\end{equation*}
And hence,
\[\sum_{r \in \mathbb{F}_q}|\widehat{\textbf{1}_{\mathcal{R}^c}}(r)| = \sum_{r \in \mathbb{F}_q\backslash 0 }|\widehat{\textbf{1}_{\mathcal{R}}}(r)| + |q-\widehat{\textbf{1}_{\mathcal{R}}}(0)| = \sum_{r \in \mathbb{F}_q}|\widehat{\textbf{1}_{\mathcal{R}}}(r)| + q-2s.\]
It therefore suffices to show that $\sum_{r \in \mathbb{F}_q}| \widehat{\textbf{\textbf{1}}_\mathcal{R}}(r)| \leq q\sqrt{s}$. By the Cauchy-Schwarz inequality,
\begin{align*}
\left( \sum_{r \in \mathbb{F}_q}|\widehat{\textbf{1}_\mathcal{R}}(r)| \right)^2 &\leq \left(\sum_{r \in \mathbb{F}_q} 1 \right)\left( \sum_{r \in \mathbb{F}_q} \left|\sum_{n \in \mathcal{R}}\psi(rn)\right|^2 \right) \\
& = q \sum_{r \in \mathbb{F}_q}\sum_{n_1, n_2 \in \mathcal{R}} \psi(r(n_1-n_2)).
\end{align*}
By swapping the order of summation we see that the total contribution from the terms with $n_1 \neq n_2$ is 0. The terms $n_1= n_2$ contribute $q^2s$ as required.
\end{proof}

The next lemma is similar to Lemma 7 from~\cite{Pol}.

\begin{lemma}  \label{l3}
Let $a, g \in \mathbb{F}_q[t]$ be coprime polynomials with $|a|<|g|$ and $g$ not a power of $t$ and let $d = \deg g > 0$. Then
\[\left| S_\mathcal{R}(a/g) \right| \leq (q-s)^{n-[\frac{n}{d}]}s^{[\frac{n}{d}]}.\]
\end{lemma}

\begin{proof}
Write $a/g = \sum_{i<0}x_i t^i$ and let $z$ be the number of non-zeros amongst the $x_i$ in the range $-n\leq i \leq -1$. Then, by our expression for $S_\mathcal{R}(a/q)$ from the start of the proof of Lemma~\ref{l2} we have that \[|S_\mathcal{R}(a/g)|=(q-s)^{n-z} \prod_{\substack{i=0 \\ x_{-i-1} \neq 0}}^{n-1}\left|\sum_{n_i \in \mathcal{R}}\psi(n_ix_{-i-1})\right| \leq (q-s)^{n-z}s^z\] by the triangle inequality. Since $q-s\geq s$, it suffices to show that $z \geq [\frac{n}{d}].$ We use proof by contradiction. Suppose $z \leq [\frac{n}{d}] -1.$ Then, by the pigeonhole principle, there is some string of at least $d$ consecutive zeros in $(x_{-n}, \ldots, x_{-1}).$ Hence, $|\{t^r a/g\}| \leq 1/q^{d+1}$ for some integer $r\geq 0$ where $\{x\} = \sum_{i<0}x_it^i$ denotes the fractional part of $x$. But this is a contradiction since $g$ does not divide $t^r a$ so we must have $|\{t^r a /g \}| \geq 1/q^{d}$.
\end{proof}

\begin{lemma}  \label{l4}
For $d \leq n/2$ we have
\[\sum_{\substack{ \deg a < \deg g \leq d \\ (a,g)=1 }}\left| S_\mathcal{R}\left(\frac{a}{g} \right) \right| \leq (q-s)^{n-2d}(q(1+\sqrt{s})-2s)^{2d}.\]
\end{lemma}

\begin{proof}
For any integer $Y$ and $x \in \mathbb{T}$, define \[S_\mathcal{R}^Y(x) = \sum_{m \in \mathcal{M}_\mathcal{R}(Y)}\textbf{e}_q(mx)\]
so that $S_\mathcal{R}(x) = S_\mathcal{R}^n(x).$
Then
\begin{align*}
\left|S_\mathcal{R}^n(x)\right| &= \left|\prod_{i=0}^{n-1}\sum_{n_i \in \mathcal{R}^c }\psi(n_ix_{-i-1}) \right| \\
&= \left|\prod_{i=0}^{Y-1}\sum_{n_i \in \mathcal{R}^c}\psi(n_ix_{-i-1}) \prod_{i=Y}^{n-1}\sum_{n_i \in \mathcal{R}^c}\psi(n_ix_{-i-1})\right| \\
&= \left|S_\mathcal{R}^Y(x)S_\mathcal{R}^{n-Y}(x t^Y)\right|.
\end{align*}
Applying this with $Y = 2d$ gives
\begin{align*}
\sum_{\substack{ \deg a < \deg g \leq d \\ (a,g)=1 }}\left| S_\mathcal{R}\left(\frac{a}{g} \right) \right| &= \sum_{\substack{ \deg a < \deg g \leq d \\ (a,g)=1 }}\left|S_\mathcal{R}^{2d}\left(\frac{a}{g}\right)S_\mathcal{R}^{n-2d}\left( \frac{t^{2d}a}{g}\right)\right| \\
&\leq \max_{\substack{ \deg a < \deg g \leq d \\ (a,g)=1 }}\left| S_\mathcal{R}^{n-2d}\left(\frac{t^{2d}a}{g}\right) \right| \sum_{\substack{ \deg a < \deg g \leq d \\ (a,g)=1 }}\left|S_\mathcal{R}^{2d}\left(\frac{a}{g}\right)\right| \\
&\leq (q-s)^{n-2d}\sum_{\substack{ \deg a < \deg g \leq d \\ (a,g)=1 }}\left|S_\mathcal{R}^{2d}\left(\frac{a}{g}\right)\right|,
\end{align*}
where we have used the trivial bound $|S_\mathcal{R}^{n-2d}(x)| \leq (q-s)^{n-2d}$. Notice that $S_\mathcal{R}^{2d}(a/g + \gamma )$ is constant in the range $|\gamma|<1/q^{2d}$ and recall that the Farey arcs $\mathcal{F}(a/g, q^{2d})$ are disjoint. Therefore
\[\frac{1}{q^{2d}}\sum_{\substack{ \deg a < \deg g \leq d \\ (a,g)=1 }}\left|S_\mathcal{R}^{2d}\left(\frac{a}{g}\right)\right| = \sum_{a,q}\int_{ \mathcal{F}(a/g, q^{2d})}\left|S_\mathcal{R}^{2d}\left(\frac{a}{g}+ \gamma\right)\right|d \gamma \leq (\sqrt{s}+1-2s/q)^{2d}\]
by Corollary~\ref{c1} where the sum is over all distinct fractions $a/q$ with $\deg g \leq d$.
\end{proof}

\begin{lemma}\label{l5}
Let $g \in \mathbb{F}_q[t]$. Then \[\frac{q^{\deg g}}{\phi(g)} = \prod_{\omega|g}\left(1-\frac{1}{q^{\deg \omega}}\right)^{-1} \leq (1+\log_q(\deg g))e.\]
\end{lemma}

\begin{proof}
Arrange the monic, irreducibles $\omega_1,\ldots,\omega_r$ dividing $g$ and the monic irreducibles $P_1,\ldots$ in $ \mathbb{F}_q[t] $ in order of degree (ordering those of the same degree arbitrarily). Then we must have that $ \deg P_i \leq \deg \omega_i $. Now, for some $ N $, we have that $ \sum_{P: \deg P \leq N-1} \deg P < \deg g \leq \sum_{P: \deg P \leq N} \deg P $. This implies that $g$ has at most $\pi(N)$ irreducible factors, and so, since $ \deg P_i \leq \deg \omega_i $ we have
\[ \prod_{\omega|g} (1-q^{-\deg \omega})^{-1}  \leq \prod_{P:\deg P \leq N} (1-q^{-\deg P})^{-1} . \]
Taking the logarithm of the right hand side, and using the fact that $ -\log(1-\frac{1}{x}) \leq \frac{1}{x-1} $ for $x>1$, and that $\sum_{d|r}d\pi(d) = q^r$ so $\pi(r)r \leq q^r -1$ we get
\[\sum_{P:\deg P \leq N} -\log(1-q^{-\deg P}) \leq \sum_{r \leq N} \frac{\pi(r)}{q^{r}-1} \leq \sum_{r \leq N} \frac{1}{r} \leq 1+ \log N.\]
Now $N$ is bounded in terms of $ \deg g $ as follows,
\[\deg g > \sum_{P: \deg P \leq N-1} \deg P = \sum_{r \leq N-1} \pi(r) r \geq \sum_{r | N-1} \pi(r) r = q^{N-1},\]
and hence $ N \leq 1+\log_q \deg g$.
Combining these inequalities gives the result.
\end{proof}

\section{Proof of Theorem~\ref{thm}}
Recall that our aim is to evaluate $N(\mathcal{R},n) = \int_{\mathbb{T}}S(x)\overline{S_\mathcal{R}(x)}d x$. Now each $x \in \mathbb{T}$ can be written as $a/g + \gamma$ for unique $a,g,\gamma$ as in Lemma~\ref{l1} which allows us to write
\[ N(\mathcal{R},n) = \int_{\mathbb{T}}\overline{S_\mathcal{R}(x)}\left(\frac{\mu(g)}{\phi(g)}\pi(n)\textbf{e}_q(\gamma t^n)\textbf{1}_{|\gamma| < 1/q^n} + E\right) dx,\]
where $|E|\leq q^{n-\frac{1}{2}[\frac{n}{2}]}$ uniformly. The error term is bounded by using Corollary~\ref{c1} as
\begin{equation}
\left| \int_{\mathbb{T}}\overline{S_\mathcal{R}(x)}E dx \right| \leq q^{n-\frac{1}{2}[\frac{n}{2}]} (\sqrt{s}+1-2s/q)^n.
\end{equation}
We can write what's left as
\[\int_{\mathbb{T}}\overline{S_\mathcal{R}(x)}\frac{\mu(g)}{\phi(g)}\pi(n)\textbf{e}_q(\gamma t^n)\textbf{1}_{|\gamma| < 1/q^n} dx = \sum_{a,g} \int_{\mathcal{F}(a/g,\: q^{ n })}\overline{S_\mathcal{R}\left(\frac{a}{g}+\gamma\right)}\frac{\mu(g)}{\phi(g)}\pi(n)\textbf{e}_q(\gamma t^n) d\gamma \]
where the sum is over all distinct fractions $a/g$ such that $\deg g \leq n/2$. These are the so-called major arcs.

Since $|\gamma| < 1/q^n$, from the definition we get
\[S_\mathcal{R}\left(\frac{a}{g} + \gamma \right) = \sum_{m \in \mathcal{M}_\mathcal{R}(n) }\mathbf{e}_q(am/g)\mathbf{e}_q(m\gamma) = \mathbf{e}_q(\gamma t^n)S_\mathcal{R}\left(\frac{a}{g}\right) \]
and therefore, since the integrand is constant on each of these major arcs which have measure $1/q^n$ the contribution becomes
\begin{equation}
\frac{\pi(n)}{q^n}\sum_{a,g} \overline{S_\mathcal{R} \left(\frac{a}{g} \right)}\frac{\mu(g)}{\phi(g)}.
\end{equation}

Let us first analyse the terms with $g=1$ and $g=t$, that is, look at
\[M=\frac{\pi(n)}{q^n}\left(S_\mathcal{R}(0)+\sum_{b \in \mathbb{F}_q\backslash 0}\overline{S_\mathcal{R} \left(\frac{b}{t} \right)}\frac{\mu(t)}{\phi(t)}\right).\]
The $g=1$ term gives $S_\mathcal{R}(0) = (q-s)^n$. Using our expression for $S_\mathcal{R}(b/t)$ from the start of the proof of Lemma~\ref{l2}, the terms $g=t$ are
\[\sum_{b \in \mathbb{F}_q\backslash 0} S_\mathcal{R}(b/t) = (q-s)^{n-1}\sum_{b \in \mathbb{F}_q\backslash 0 } \sum_{n \in \mathcal{R}^c } \mathbf{e}_q(\frac{n b}{t}) = -(q-s)^{n-1}\sum_{b \in \mathbb{F}_q\backslash 0 } \sum_{r \in \mathcal{R} } \psi(br).\]
Using
\[\sum_{b \in \mathbb{F}_q\backslash 0}\psi(br) = 
\begin{cases}
q-1 &\text{ if } r=0 \\
-1 &\text{ if } r \neq 0
\end{cases}
\]
This becomes
\[\begin{cases}
-(q-s)^{n} & \text{ if } 0 \in \mathcal{R} \\
(q-s)^{n-1}s & \text{ if } 0 \notin \mathcal{R}.
\end{cases}\]
Hence, since $\mu(t) = -1$ and $\phi(t) = q-1$ we have
\begin{align*}
M &= \frac{\pi(n)}{q^n} \left( (q-s)^n -\frac{1}{q-1} \sum_{b \in \mathbb{F}_q\backslash 0}S_\mathcal{R}(b/t)\right) \\
& =  \frac{q\Lambda}{q-1}\pi(n)(1-s/q)^n,
\end{align*}
where
\[\Lambda = \begin{cases}
1 & \text{ if } 0 \in \mathcal{R} \\
1-\frac{1}{q-s} & \text{ if } 0 \notin \mathcal{R}.
\end{cases}\]
Using $\pi(n) \leq q^n/n$, the remaining terms in (4.2) are bounded by

\[\frac{1}{n}\sum_{\substack{ 1\leq \deg g \leq n/2 \\ g\neq t }}\frac{|\mu(g)|}{\phi(g)} \sum_{\substack{\deg a < \deg g \\ (a,g)=1 }}\left| S_\mathcal{R} \left(\frac{a}{g}\right)\right|.\]
Let $U$ be some parameter $1 \leq U \leq n/2$ to be specified shortly. Grouping the $g$ according to their degree and using Lemma~\ref{l3} for the terms with $d = \deg g \leq U $ and Lemmas~\ref{l4} and~\ref{l5} for the terms with $\deg g > U$ we get
\begin{align*}
&\sum_{\substack{1\leq  \deg g \leq n/2\\ g\neq t }}\frac{|\mu(g)|}{\phi(g)} \sum_{\substack{\deg a < \deg g \\ (a,g)=1 }}\left| S_\mathcal{R} \left(\frac{a}{g}\right)\right| \\
&\leq \sum_{1\leq d\leq U}q^d(q-s)^{n - [\frac{n}{d}]} s^{[\frac{n}{d}]} + e\sum_{ U  < d \leq n/2 }q^{-d}(q-s)^{n-2d}(q(1+\sqrt{s})-2s)^{2d}(1 + \log_q(d))\\
& =(q-s)^n \left( \sum_{1 \leq d\leq U} q^d\left(\frac{s}{q-s}\right)^{[\frac{n}{d}]}  + e\sum_{U < d \leq n/2} q^d\left(\frac{1+\sqrt{s}-2s/q}{q-s}\right)^{2d}(1 + \log_q(d)) \right) \\
& \ll (q-s)^n\left( n \left(q^U\left(\frac{s}{q-s}\right)^{n/U} + q^{U/2}\left(\frac{\sqrt{s}+1-2s/q}{q-s}\right)^{U} \right) \right)
\end{align*}
We have trivially bounded the first sum. The bound for the second sum follows after using $1+ \log_q(d) \leq n $ and bounding the resulting geometric sum using $s\leq \sqrt{q}/2$ so that
\[\frac{\sqrt{q}(\sqrt{s}+1-2s/q)}{q-s} \leq \frac{q/2+\sqrt{q}}{q-\sqrt{q}/2}<0.99\] for $q \geq 11$. Taking $U = (2n/5)^{1/2}$ and using $s \leq \sqrt{q}/2$ this becomes
\begin{align*}
&\ll(q-s)^n\left( n \left(q^{\sqrt{\frac{2}{5}n}}\left(\frac{q^{1/2}}{2q-q^{1/2}}\right)^{\sqrt{\frac{5}{2}n}} + q^{\sqrt{\frac{1}{10}n}}\left(\frac{q^{1/4}/\sqrt{2}+1}{q-q^{1/2}/2}\right)^{\sqrt{\frac{2}{5}n}} \right) \right) \\
&\ll n(q-s)^n q^{-n^{1/2}/(2\sqrt{10})},
\end{align*}
since $\sqrt{\frac{2}{5}}-\frac{1}{2}\sqrt{\frac{5}{2}} = -\frac{1}{2\sqrt{10}}$ and $\sqrt{\frac{1}{10}}-\frac{3}{4}\sqrt{\frac{2}{5}} = -\frac{1}{2\sqrt{10}}$. Combining this with our expression for the main term $M$ and error estimate (4.1) we get
\begin{equation}
N(\mathcal{R},n) =  \frac{ q }{q-1}\frac{(q-s)^n}{n}\left( \Lambda + O\left( n \mathcal{E} \right) \right)
\end{equation}
where
\begin{equation}
\mathcal{E} \ll q^{-n^{1/2}/(2\sqrt{10})} + \left(\frac{q^{3/4}(s^{1/2}+1)}{q-s}\right)^n.
\end{equation}
Since $s \leq \sqrt{q}/2$, we then have
\begin{equation}
\mathcal{E} \ll q^{-n^{1/2}/(2\sqrt{10})} + \left(\frac{q/\sqrt{2} + q^{3/4}}{q-\sqrt{q}/2}\right)^n.
\end{equation}
A calculation reveals that for $n \geq 100(\log q)^2$, the first expression is larger then the second when $q \geq 500$ and that both are $ \ll q^{-n^{1/2}/7}/n$ which completes the proof of Theorem~\ref{thm}.

\begin{remark}
The conditions on the sizes of $s$, $q$ and $n$ were made in order to simplify the statement of Theorem~\ref{thm} but (4.4) is also interesting for other choices. For example, when $n$ is fixed, we have that $\mathcal{E} \rightarrow 0 $ as $q \rightarrow \infty$ provided $s = o(q^{1/2}).$

Recall that in the special case $s=1$, we have equality in Corollary~\ref{c1}. Feeding this through the rest of the proof gives
\[\mathcal{E} \ll q^{-n^{1/2}/(2\sqrt{10})} + \left( \frac{q^{3/4}(2-2/q)}{q-1} \right)^n.\]
For $q\geq 17$, the expression in the brackets is less than 1 which proves $n\mathcal{E} \rightarrow 0$ as $n \rightarrow \infty$ in this case.
\end{remark}

\section{Proof of Theorem~\ref{thm2}}

Our proof of Theorem~\ref{thm2} is the same as Theorem~\ref{thm} except that we use modified versions of Corollary~\ref{c1} and Lemma~\ref{l3} which we will now prove. In this section, we assume that $p$ is a prime, $\mathcal{R} \subset \mathbb{F}_p$ is subset of consecutive coefficients and use the same notation already introduced.

\begin{corollary}\label{c2}
\[\int_{\mathbb{T}}|S_\mathcal{R}(x)|dx \leq (\log p + 1 - s/p)^n.\]
\end{corollary}

\begin{proof}
Write $\mathcal{R} = \{d, d+1, \ldots, d+s-1 \}$, then, if $r=0$, $|\widehat{\textbf{1}_{\mathcal{R}^c}}(r)| = p-s$, and if $r \neq 0$,
\[|\widehat{\textbf{1}_{\mathcal{R}^c}}(r)| = \left| \sum_{k=d}^{d+s-1}e^{2\pi i kr/p} \right| = \left|\frac{1-e^{2 \pi i sr/p}}{1-e^{2\pi i r/p}}\right| \leq \frac{1}{|\sin\pi r/p|}.\]
Therefore,
\[\sum_{r \in \mathbb{F}_p}|\widehat{\textbf{1}_{\mathcal{R}^c}}(r)| \leq p-s + \sum_{r=1}^{p-1}\frac{1}{|\sin \pi r/p|}< p-s+2\sum_{r=1}^{\frac{p-1}{2}}\frac{p}{2 r}<p-s+p\log p .\]
Now use Lemma~\ref{l2}.
\end{proof}

Consequently, the bound in Lemma~\ref{l4} is replaced by $(p-s)^{n-2n}(p(\log p + 1) - s)^{2n}.$

\begin{lemma}\label{l6}
Let $a, g \in \mathbb{F}_p[t]$ be coprime polynomials with $|a|<|g|$ and $g$ not a power of $t$ and let $d = \deg g > 0$. Then
\[\left| S_\mathcal{R}(a/g) \right| \leq (p-s)^{n}e^{-[\frac{n}{d}]\frac{1}{p^3}}.\]
\end{lemma}

\begin{proof}
As in the proof of Lemma~\ref{l3} we have \[|S_\mathcal{R}(a/g)|=(p-s)^{n-z} \prod_{\substack{i=0 \\ x_{-i-1} \neq 0}}^{n-1}\left|\sum_{n_i \in \mathcal{R}}e^{2\pi i(n_ix_{-i-1})/p}\right|.\]
For $x \in \mathbb{F}_p \backslash \{0\}$, we have
\[\left|e^{2 \pi i \frac{x}{p}n} + e^{2 \pi i \frac{x}{p}(n+1)}\right|^2 = 2+2\cos(\frac{2\pi x}{p}) < 4 e^{-2/p^2},\]
and therefore
\[\left|\sum_{n_i \in \mathcal{R}}e^{2\pi i(n_ix_{-i-1})/p}\right| \leq p-s-2+2e^{-1/p^2} \leq (p-s)e^{-1/p^3}.\]
Recalling from the proof of Lemma~\ref{l3} that $z \geq [n/d]$ completes the proof.
\end{proof}

Provided $p$ is large enough to ensure that $\frac{\sqrt{p}(\log p + 1 - s/p)}{p-s}<1$ (so the resulting geometric sum we saw earlier converges) we may just insert these new bounds into the proof of Theorem~\ref{thm} to get (4.3) with
\[\mathcal{E} \ll p^U e^{-[\frac{n}{U}]\frac{1}{p^3}} + \left(\frac{\sqrt{p}(\log p + 1 - s/p)}{p-s}\right)^{U} + \left(\frac{p^{3/4}(\log p + 1 - s/p)}{p-s}\right)^n \]
for some parameter $U$. Taking $U = cn^{1/2}$, and since we are assuming $p-s > p^{3/4+\delta}$, this proves Theorem~\ref{thm2} for some $c>0$ sufficiently small in terms of $p$ and $\delta$.

\section*{Acknowledgements}\label{ackref}
The author would like to thank his supervisor, Andrew Granville, for his encouragement and suggestions and Simon Myerson for his helpful comments. This work was supported by the Engineering and Physical Sciences Research Council EP/L015234/1 via the EPSRC Centre for Doctoral Training in Geometry and Number Theory (The London School of Geometry and Number Theory), University College London.


\begin{thebibliography}{HD}

\bibitem{DaMaS}
C. Dartyge, C. Mauduit and A. S\'{a}rk\"{o}zy,
\textit{Polynomial values and generators with missing
digits in finite fields},
Funct. Approx. Comment. Math.
,
52
(2015), 65–74.

\bibitem{DiElSh}
R. Dietmann, C. Elsholtz, I. Shparlinski,
\textit{Prescribing the binary digits of squarefree
numbers and quadratic residues},
 Trans. AMS, (2017).
 
\bibitem{Ha}
J. Ha,
\textit{Irreducible polynomials with several prescribed coefficients},
Fin. Fields App.
40 (2016) 10-25.

\bibitem{Hay}
D. R. Hayes,
\textit{The expression of a polynomial as a sum of three irreducibles},
Acta Arith. 11 (1966) 461–488.

\bibitem{May}
J. Maynard,
\textit{Primes with restricted digits}, preprint, arDiv:1604.01041, 2016.
 
\bibitem{OpSh}
A. Oppenheim, M. Shusterman,
\textit{Squarefree polynomials with prescribed coefficients},
preprint, arXiv:1707.06528, 2017.

\bibitem{Pol}
P. Pollack,
\textit{Irreducible polynomials with several prescribed coefficients},
Finite Fields Appl. 22 (2013) 70–78.
 
\bibitem{TuWa}
A. Tuxanidy, Q. Wang,
\textit{Irreducible polynomials with prescribed sums of coefficients},
preprint,  arXiv:1605.00351, 2016.

\end{thebibliography}
\end{document}